\documentclass[12pt]{amsart}
\usepackage{latexsym,amsmath,amsthm,amssymb,amscd,amsfonts}
\usepackage[usenames]{color}

\newtheorem{lemma}{Lemma}[section]
\newtheorem{proposition}[lemma]{Proposition}
\newtheorem{remark}[lemma]{Remark}
\newtheorem{theorem}[lemma]{Theorem}
\newtheorem{definition}[lemma]{Definition}
\newtheorem{claim}[lemma]{Claim}

\newtheorem{corollary}[lemma]{Corollary}

\newtheorem{ques}[lemma]{Question}

\newtheorem*{remark*}{Remark}

\def\supp{\text{supp}}

\makeatletter \@addtoreset {equation}{section}

\renewcommand\theequation
  {\ifnum \c@subsection>\z@ \arabic{section}.\arabic{subsection}.\arabic{equation}
  \else \arabic{section}.\arabic{equation} \fi}
\makeatother

\begin{document}

\title[Near commutativity vs. almost commutativity in symplectic category]{The gap between near commutativity and almost commutativity in symplectic category}

\author[Lev Buhovsky]{Lev Buhovsky$^{1}$}
\footnotetext[1]{The author also uses the spelling ``Buhovski"
for his family name.}
\footnotetext[2]{First published in Electronic Research Announcements in Mathematical Sciences in Volume 20 (2013), pp 71-76, published by American Institute of Mathematical Sciences (AIMS). \\ $ \copyright $ 2013 American Institute of Mathematical Sciences.}
\address{Lev Buhovski \\ School of Mathematical Sciences, Tel Aviv University}
\email{levbuh@post.tau.ac.il}
\thanks{I would like to thank Leonid Polterovich for constant encouragement, interest in this work, useful comments, and for finding a mistake in the first draft of the paper. This work was partially supported by the Israel Science Foundation grant 1380/13.}
\date{\today}

\begin{abstract}
On any closed symplectic manifold of dimension greater than $ 2 $, we construct a pair of smooth functions, such that on the one hand, the uniform norm of their Poisson bracket equals to $ 1 $, but on the other hand, this pair cannot be reasonably approximated (in the uniform norm) by a pair of Poisson commuting smooth functions. This comes in contrast with the dimension $ 2 $ case, where by a partial case of a result of Zapolsky~\cite{Z-2}, an opposite statement holds. 
\end{abstract}

\maketitle


\section{Introduction and results}

During the last ten years, function theory on symplectic manifolds has attracted a great deal of attention~\cite{B,BEP,CV,EP-qst,EP-Poisson,EP-Poisson2,EPR,EPZ,P,Z-1,Z-2}. The $ C^0 $-rigidity of the Poisson bracket~\cite{EP-Poisson} (cf.~\cite{B}) is one of the achievements of this theory, and it states that on a closed symplectic manifold $ (M,\omega) $, the uniform norm of the Poisson bracket of a pair of smooth functions on $ M $ is a lower semi-continuous functional, when we consider the uniform (or the $ C^0 $) topology on $ C^{\infty}(M) \times C^{\infty}(M) $. Informally speaking, this means that one cannot significantly reduce the $ C^0 $ norm of the Poisson bracket of two smooth functions by an arbitrarily small $ C^0 $ perturbation. The $ C^0 $-rigidity of the Poisson bracket holds also when the symplectic manifold $ (M,\omega) $ is open, if we restrict to smooth compactly supported functions. In this context it is natural to ask, how strongly one should perturb a given pair of smooth functions in order to significantly reduce the $ C^0 $ norm of their Poisson bracket. The following question was asked privately by Polterovich in 2009 (later it also appeared in~\cite{Z-2}):

\begin{ques} \label{Q:Polterovich}
Let $ (M,\omega) $ be a closed symplectic manifold. Does there exist a constant $ C > 0 $, such that for any pair of smooth functions $ f,g : M \rightarrow \mathbb{R} $ satisfying $ \| \{ f,g \} \| = 1 $, there exists a pair of smooth functions $ F,G : M \rightarrow \mathbb{R} $, such that
$ \| F - f \|, \| G - g \| \leqslant C $ and such that $ \{ F,G \} = 0 $ on $ M $?
\end{ques}

The $ 2 $-dimensional case of this question was answered affirmatively by Zapolsky, 
and in fact, it appeared as a particular case of a more general statement which applies to functions on a manifold with a volume form~\cite{Z-2}.

The main result of this note is

\begin{theorem} \label{T:solution-question-P}
Let $ (M,\omega) $ be a symplectic manifold of dimension $ 2n > 2 $. Then for any $ 0 < p < q $ there exists a pair of smooth compactly supported functions $ f,g : M \rightarrow \mathbb{R} $, such that $ \| f \| = \| g \| = 1 $, $ \| \{ f,g \} \| = q $, and such that for any $ s \in [0,q] $ we have 

\begin{equation} \label{E:1}
\frac{1}{2} - \frac{1}{2p} s \leqslant \rho_{f,g}(s) \leqslant \frac{1}{2} - \frac{1}{2q} s
\end{equation}
\end{theorem}

\noindent Here $ \rho_{f,g}(s) $ is the {\em profile function}~\cite{BEP}, which is defined as
$$ \rho_{f,g}(s) = \inf \{ \| F-f \| + \| G-g \| \,\, | \,\, F,G \in C^{\infty}_c(M) \,\, \text{and} \,\, \| \{ F,G\} \| \leqslant s \} .$$

Now let $ (M,\omega) $ be a symplectic manifold of dimension $ 2n > 2 $. Note that for a pair $ f,g : M \rightarrow \mathbb{R} $ of smooth compactly supported functions, the value $ \rho_{f,g}(0) $ is precisely the uniform distance from the pair $ f,g $ of functions to the set of pairs $ F,G : M \rightarrow \mathbb{R} $ of smooth compactly supported Poisson commuting functions. Hence, if we take any $ C > 0 $, and set $ q = 1/(64C^2) $, then Theorem~\ref{T:solution-question-P} implies the existence of a pair $ f,g : M \rightarrow \mathbb{R} $ of smooth compactly supported functions, such that $ \| \{ f,g \} \| = q = 1/(64C^2) $ and $ \rho_{f,g} (0) = 1/2 $, and hence for the functions $ f_1 = 8C f $, $ g_1 = 8C g $ we first of all get $ \| \{ f_1 , g_1 \} \| = 1 $, and moreover we get $ \rho_{f_1,g_1}(0) = 4C $, which implies that there does not exist a  pair $ F,G : M \rightarrow \mathbb{R} $ of smooth compactly supported Poisson commuting functions, such that $ \| F-f_1 \| , \| G-g_1 \| \leqslant C $. Thus, we obtain 

\begin{corollary}
The answer to Question~\ref{Q:Polterovich} is negative in dimension $ > 2 $.
\end{corollary}

The proof of Theorem~\ref{T:solution-question-P} relies on a certain computation of the $ pb_4 $ invariant, which is carried out similarly as an analogous computation in the proof of Proposition 1.21 in~\cite{BEP}. Some part of the proof of Theorem~\ref{T:solution-question-P} is reminiscent of the proof of Theorem 1.4 (ii) in~\cite{BEP}. 

\begin{remark}
Question~\ref{Q:Polterovich} has an analogue in the setting of Hermitian matrices equipped with a commutator (see~\cite{PS,H}). 
\end{remark}

%


\section{Proofs}

Let us first remind the definition of the $ pb_4 $ invariant which was initially introduced in~\cite{BEP} and which participates in the proof of Theorem~\ref{T:solution-question-P}. 

\begin{definition} \label{D:pb_4}
Let $ (M,\omega) $ be a symplectic manifold. Let $ X_0,X_1,Y_0,Y_1 \subset M $ be a quadruple of compact sets. Then we define $$ pb_4(X_0,X_1,Y_0,Y_1) = \inf \| \{ F,G \} \| ,$$
where the infimum is taken over the set of all pairs of smooth compactly supported functions $ F,G : M \rightarrow \mathbb{R} $ such that $ F \leqslant 0 $ on $ X_0 $, $ F \geqslant 1 $ on $ X_1 $, $ G \leqslant 0 $ on $ Y_0 $, and $ G \geqslant 1 $ on $ Y_1 $. Note that such set of pairs of functions is non-empty whenever $ X_0 \cap X_1 = Y_0 \cap Y_1 = \emptyset $. 
If the latter condition is violated, we put $$ pb_4(X_0,X_1,Y_0,Y_1) = +\infty .$$
\end{definition}

The following result was proved in~\cite{BEP} (in~\cite{BEP} this is Proposition 1.21):

\begin{proposition} \label{P:prop-from-BEP}
Let $ (M,\omega) $ be a symplectic surface of area $ B > 0 $. Consider a curvilinear quadrilateral $ \Pi \subset M $ of area $ A > 0 $ with sides denoted in the cyclic order by $ a_1,a_2,a_3,a_4 $ - that is $ \Pi $ is a topological disc bounded by the union of four smooth embedded curves $ a_1,a_2,a_3,a_4 $ connecting four distinct points in $ M $ in the cyclic order as listed here and (transversally) intersecting each other only at their common end-points. 
Let $ L $ be an exact section of $ T^* S^1 $. Assume that $ M \neq S^2 $ and that $ 2A \leqslant B $. Then in the symplectic manifold $ M \times T^* S^1 $ (with the split symplectic form) we have $$ pb_4(a_1 \times L, a_3 \times L, a_2 \times L, a_4 \times L) = 1/A > 0 .$$ 
\end{proposition}

The proof of Proposition~\ref{P:prop-from-BEP}, which is presented in~\cite{BEP}, is divided into two parts. The first part proves the inequality $ pb_4(a_1 \times L, a_3 \times L, a_2 \times L, a_4 \times L)  \leqslant 1/A $ by providing a concrete example of a pair of functions $ F,G $ as in Definition~\ref{D:pb_4}. The second part proves the inequality $ pb_4(a_1 \times L, a_3 \times L, a_2 \times L, a_4 \times L)  \geqslant 1/A $, and it uses the Gromov's theory of pseudo-holomorphic curves.

Now we turn to the proofs of our results.
\begin{lemma} \label{L:curve-existence}
Consider $ \mathbb{R}^{2n} $ endowed with the standard symplectic structure $ \omega_{std} $. Then given any open set $ U \subseteq \mathbb{R}^{2n} $ and any $ A > 0 $, there exists a smooth symplectic embedding 
$ (B^2(A), \sigma_{std}) \rightarrow (U,\omega_{std}) $, where $ B^2(A) \subset \mathbb{R}^2 $ is the disc of area $ A $ centred at the origin, and $ \sigma_{std} $ is the standard area form on $ \mathbb{R}^2 $ (and on $ B^2(A) $). 
\end{lemma}
\begin{proof} It is enough to construct a smooth embedding of $ B^{2}(\pi) $ into $ B^{2n}(2\pi) $ (where $ B^{2n}(2\pi) $ is the $ 2n $-dimensional ball of capacity $ 2\pi $, or equivalently, radius $ \sqrt{2} $, centred at the origin), such that its image is a symplectic curve having arbitrarily large symplectic area. An example of such an embedding is the map $ u : B^{2}(\pi) \rightarrow \mathbb{C}^n $ given by $ u(z) = (z^k,z, 0 ,0 , ..., 0) $, where $ k \in \mathbb{N} $ is large enough.
\end{proof}

\begin{lemma} \label{L:pb4}
Assume that we have a curvilinear quadrilateral of area $ A > 0 $ on the plane $ \mathbb{R}^2 $, with sides $ a_1, a_2, a_3, a_4 $ (written in the cyclic order), and let $ L \subset T^* \mathbb{T}^{n-1} $ be an exact section. Then in the symplectic manifold $ \mathbb{R}^2 \times T^* \mathbb{T}^{n-1} $ (with the split symplectic form) we have $$ pb_4(a_1 \times L, a_3 \times L, a_2 \times L, a_4 \times L) = 1/A > 0 .$$
\end{lemma}
\begin{proof}
The proof goes similarly as the proof of Proposition~\ref{P:prop-from-BEP} (which is Proposition 1.21 in~\cite{BEP}).
\end{proof}

\begin{proof}[Proof of Theorem~\ref{T:solution-question-P}] 
Consider a Darboux neighborhood $ U \subset M $. Denote $ l = 1 / \sqrt{q} $, $ \epsilon = 1/p - 1/q $. Then $ q = 1 / l^2 $, $ p = 1 / (l^2 + \epsilon) $. By Lemma~\ref{L:curve-existence}, there exists a symplectic embedding $ u : (B^2(2l^2+3\epsilon),\sigma_{std}) \rightarrow (U,\omega) $. By the symplectic neighborhood theorem, a neighborhood of $ u (B^2(2l^2+3\epsilon)) $ contains an open subset which is symplectomorphic to the product $ B^2(2l^2+2\epsilon) \times B^{2}(\epsilon')^{\times (n-1)} $ of a disc of area $ 2l^2 + 2\epsilon $ and $ n-1 $ copies of a disc of area $ \epsilon' $, for some $ \epsilon' > 0 $. Let $ \phi : (W = B^2(2l^2+2\epsilon) \times B^{2}(\epsilon')^{\times (n-1)},\omega_{std}) \rightarrow (M,\omega) $ be a symplectic embedding whose image is this open subset. Choose a curvilinear quadrilateral inside $ B^2(2l^2+2\epsilon) $ with sides $ a_1,a_2,a_3,a_4 $ (which are given in cyclic order), and area $ l^2+\epsilon $. 
 
 \begin{claim} \label{Cl:1}
 There exist smooth compactly supported functions $ \tilde{f}_1,\tilde{g}_1 : B^2(2l^2+2\epsilon) \rightarrow \mathbb{R} $, such that $ \| \{ \tilde{f}_1,\tilde{g}_1 \} \| = 1/l^2 = q $, such that $ \tilde{f}_1 = 0 $ on $ a_1 $, $ \tilde{f}_1 = 1 $ on $ a_3 $, $ \tilde{g}_1 = 0 $ on $ a_2 $, $ \tilde{g}_1 = 1 $ on $ a_4 $, such that $ \tilde{f}_1,\tilde{g}_1 \geqslant 0 $ on $ B^2(2l^2+2\epsilon) $ and such that $ \| \tilde{f}_1 \| = \| \tilde{g}_1 \| = 1 $. 
 \end{claim}
\begin{proof}
 Denote $ \epsilon_1 = \epsilon / l > 0 $, and consider $ \epsilon_2 > 0 $ such that $ (2l + \epsilon_1 + 3\epsilon_2) (l + 3 \epsilon_2) = 2l^2 + 2\epsilon $. Denote $ \epsilon_3 = \min(\epsilon_1,\epsilon_2) / 2 $. It is enough to find smooth compactly supported functions $ \hat{f}_1, \hat{g}_1 : (-\epsilon_2,2l+\epsilon_1+2\epsilon_2) \times (-\epsilon_2,l+2\epsilon_2) \rightarrow \mathbb{R} $, such that  $ \| \{ \hat{f}_1,\hat{g}_1 \} \| = 1/l^2 $, such that $ \hat{f}_1 = 0 $ on $ \{0\} \times [0,l] $, $ \hat{f}_1 = 1 $ on $ \{ l+\epsilon_1 \} \times [0,l] $, $ \hat{g}_1 = 0 $ on $ [0,l+\epsilon_1] \times \{ 0 \} $, $ \hat{g}_1 = 1 $ on $ [0,l+\epsilon_1] \times \{ l \} $, such that $ \hat{f}_1,\hat{g}_1 \geqslant 0 $ on $ (-\epsilon_2,2l+\epsilon_1+2\epsilon_2) \times (-\epsilon_2,l+2\epsilon_2) $ and such that $ \| \hat{f}_1 \| = \| \hat{g}_1 \| = 1 $. We give an explicit construction. First, choose smooth functions $ u_1,v_1,u_2,v_2 : \mathbb{R} \rightarrow [0,1] $ with the following properties: 

\begin{enumerate} 
\item[(a)] $ \supp(u_1) \subset (0,2l+\epsilon_1+\epsilon_2) $, $ u_1(l+\epsilon_1) = 1 $, $ \| u_1' \| = 1/( l + \epsilon_3) $. 
\item[(b)] $ \supp(v_1) \subset (-\epsilon_2,l+\epsilon_2) $, $ v_1(y) = 1 $ on $ [0,l] $. 
\item[(c)] $ \supp(u_2) \subset (-\epsilon_2,2l+\epsilon_1+2\epsilon_2) $, $ u_2(x) = 1 $ on $ [0,2l+\epsilon_1+\epsilon_2] $. 
\item[(d)] $ \supp (v_2) \subset (0,l +  2\epsilon_2) $, $ v_2(l) = 1 $, \\ $ \max_{[0,l]} |v_2'(y)| = \max_{[0,l+\epsilon_2]} |v_2'(y)| = (l + \epsilon_3)/l^2 $. 
\end{enumerate}

Now define $ \hat{f}_1,\hat{g}_1 : (-\epsilon_2,2l+\epsilon_1+2\epsilon_2) \times (-\epsilon_2,l+2\epsilon_2) \rightarrow \mathbb{R} $ by $ \hat{f}_1(x,y) = u_1(x)v_1(y) $, $ \hat{g}_1(x,y) = u_2(x)v_2(y) $. Then $$ \{ \hat{f}_1 , \hat{g}_1 \} (x,y) = v_1(y)u_2(x)u_1'(x)v_2'(y) - u_1(x)v_2(y)u_2'(x)v_1'(y) $$ $$ = v_1(y)u_2(x)u_1'(x)v_2'(y) = u_1'(x)v_2'(y)v_1(y) .$$ We have $ \| u_1' \| = 1 / (l+\epsilon_3) $, and $$ \| v_2' v_1 \| = \max |v_2'(y)v_1(y)| = \max_{[0,l+\epsilon_2]} |v_2'(y)v_1(y)| = (l + \epsilon_3) / l^2 ,$$ and hence $ \|  \{ \hat{f}_1 , \hat{g}_1 \} \| = \| u_1' \| \cdot \| v_2' v_1 \| = 1/l^2 $. The rest of the claimed properties of $ \hat{f}_1,\hat{g}_1 $ follow immediately.
\end{proof}
 
Consider functions $ \tilde{f}_1,\tilde{g}_1 $, as in Claim~\ref{Cl:1}. Choose $ 0 < \rho_1 < \rho <\rho_2 $ such that $ \pi \rho_2^2 < \epsilon' $, and choose a smooth function $ \tilde{h}_1 : B^2(\epsilon') \rightarrow \mathbb{R} $ such that its support lies inside the annulus $ D^2_{\rho_1,\rho_2} = \{ z \in B^{2}(\epsilon') \, | \, \rho_1 < |z| < \rho_2 \} $, such that $ 0 \leqslant \tilde{h}_1 \leqslant 1 $ on $ B^2(\epsilon') $, and such that $ \tilde{h}_1 = 1 $ on $ S^1_\rho = \{ z \in B^2(\epsilon') \, | \, |z| = \rho \} $. Now define $ f_1, g_1 : W \rightarrow \mathbb{R} $ by $$ f_1(z,w_1,...,w_{n-1}) = 
\tilde{f}_1(z) \tilde{h}_1(w_1) ... \tilde{h}_1(w_{n-1}) ,$$ $$ g_1(z,w_1,...,w_{n-1}) = 
\tilde{g}_1(z) \tilde{h}_1(w_1) ... \tilde{h}_1(w_{n-1}) ,$$ for $ (z,w_1,...,w_{n-1}) \in W $, and then define $ f,g : M \rightarrow \mathbb{R} $ by setting $ f = g = 0 $ on $ M \setminus \phi(W) $ and setting $ f(\phi(x)) = f_1(x) $, $ g(\phi(x)) = g_1(x) ,$ for any $ x \in W $.  First note that $ f_1 = 0 $ on $ a_1 \times (S^1_\rho)^{n-1} $, $ f_1 = 1 $ on $ a_3 \times (S^1_\rho)^{n-1} $, $ g_1 = 0 $ on $ a_2 \times (S^1_\rho)^{n-1} $, and $ g_1 = 1 $ on $ a_4 \times (S^1_\rho)^{n-1} $. Secondly, we have $$ \{ f,g \}(\phi(z,w_1,...,w_{n-1})) = \{ \tilde{f}_1, \tilde{g}_1 \}(z) \tilde{h}_1^2(w_1) ... \tilde{h}_1^2(w_{n-1}) $$ for any $ (z,w_1,...,w_{n-1}) \in W $, and hence $ \| \{ f,g \} \| = q $. Also we have that $ f,g \geqslant 0 $ on $ M $, and that $ \| f \| = \| g \| = 1 $. Let us show that the constructed functions $ f,g $ satisfy ($\ref{E:1}$). For showing the upper bound, choose a smooth compactly supported function $ h : M \rightarrow [0,1] $ such that $ h = 1 $ on the union of supports $ \supp(f) \cup \supp(g) $, and define functions $ F,G : M \rightarrow \mathbb{R} $ by $ F = ( \frac{1}{2}-\frac{s}{2q} + \frac{s}{q} f) h $, $ G = g $ (the upper bound in ($\ref{E:1}$) is essentially the statement of Theorem 1.4 (ii), inequality (8) in~\cite{BEP}). Let us show the lower bound, for a given value of $ s $. Since the profile function $ \rho_{f,g} $ is non-negative and non-increasing, it is sufficient to prove the lower bound for $ s \in (0,p) $. Take $ s \in (0,p) $ and denote $ t = \frac{1}{2} - \frac{1}{2p} s > 0 $. Now assume that we have a pair of smooth compactly supported functions $ F,G \in C^\infty_c(M) $, such that $ \| F - f \| + \| G - g \| \leqslant t $. Denote $ \alpha = \| F-f \| $,
$ \beta = \| G-g \| $. Then $ \alpha, \beta \geqslant 0 $ and $ \alpha + \beta \leqslant t $. Now choose $ \delta > 0 $ small, and pick smooth functions $ u,v : \mathbb{R} \rightarrow \mathbb{R} $, such that the function $ u $ satisfies $ u(x) = 0 $ for $ x \in [-\alpha,\alpha] $, $ | u(x) - x | \leqslant (1+\delta) \alpha $ for $ x \in \mathbb{R} $, and $ |u'(x)| \leqslant 1 $ for $ x \in \mathbb{R} $, and such that the function $ v $ satisfies $ v(x) = 0 $ for $ x \in [-\beta,\beta] $, $ | v(x) - x | \leqslant (1+\delta) \beta $ for $ x \in \mathbb{R} $, and $ |v'(x)| \leqslant 1 $ for $ x \in \mathbb{R} $. Then the functions $ F ' = u \circ F $, $ G ' = v \circ G $ satisfy $ \| \{ F',G' \} \| \leqslant  \| \{ F, G \} \| $. But moreover, we have that $ F'(x) = 0 $ whenever $ f(x) = 0 $, $ F'(x) \geqslant 1 - (2 + \delta) \alpha $ whenever $ f(x) = 1 $, $ G'(x) = 0 $ whenever $ g(x) = 0 $, and $ G'(x) \geqslant 1 - (2+\delta) \beta $ whenever $ g(x) = 1 $, for any $ x \in M $. Hence if we denote $ W' =  B^2(2l^2+2\epsilon) \times (D^{2}_{\rho_1,\rho_2})^{\times (n-1)} \subset W $, then the supports of $ F',G' $ lie inside $ \phi(W') $, so in particular the supports of $ \phi^*F', \phi^* G' $ lie inside $ W' $, and moreover we have that $ \phi^* F' = 0 $ on $ a_1 \times (S^1_\rho)^{n-1} $, $ \phi^* F' \geqslant 1 - (2+\delta) \alpha $ on $ a_3 \times (S^1_\rho)^{n-1} $, $ \phi^* G' = 0 $ on $ a_2 \times (S^1_\rho)^{n-1} $, and $ \phi^* G' \geqslant 1 - (2+\delta) \beta $ on $ a_4 \times (S^1_\rho)^{n-1} $. Consider a symplectic embedding $ \psi $ of $ W' = B^2(2l^2+2\epsilon) \times (D^{2}_{\rho_1,\rho_2})^{\times (n-1)} $ into $ \mathbb{R}^2 \times (T^* S^1)^{\times (n-1)} = \mathbb{R}^2 \times T^*  \mathbb{T}^{n-1}  $ (where $ \mathbb{T}^{n-1} = (S^1)^{n-1} $ is the $ (n-1) $-dimensional torus), given as a product of maps, such that at the factor $ B^2(2l^2 + 2\epsilon) $ we have the standard embedding into $ \mathbb{R}^2 $ (given by the identity map), and such that at each factor $ D^{2}_{\rho_1,\rho_2} $, the circle $ S^1_\rho $ is mapped onto the zero section of $ T^* S^1 $.  Denote the push-forwards $ F_2 = \psi_* \phi^* F' = F' \circ \phi \circ \psi^{-1} $, $ G_2 = \psi_* \phi^* G' = G' \circ \phi \circ \psi^{-1} $, which are a priori defined on $ \psi(W') $, and extend them by $ 0 $ to obtain functions on the whole $ \mathbb{R}^2 \times T^*  \mathbb{T}^{n-1} $. We have $ F_2 = 0 $ on $ a_1 \times L_0 $, $ F_2 \geqslant 1 - (2+\delta) \alpha $ on $ a_3 \times L_0 $, $ G_2 = 0 $ on $ a_2 \times L_0 $, and $ G_2 \geqslant 1 - (2+\delta) \beta $ on $ a_4 \times L_0 $, where $ L_0 \subset T^* \mathbb{T}^{n-1} $ is the zero section. Therefore in view of Lemma~\ref{L:pb4}, we get $$  \| \{ F, G \} \| \geqslant \| \{ F',G' \} \| = \| \{ F_2 ,G_2 \} \| $$ $$ \geqslant (1 - (2+\delta) \alpha)(1 - (2+\delta) \beta) \cdot pb_4(a_1 \times L_0, a_3 \times L_0, a_2 \times L_0, a_4 \times L_0) $$ $$ = \frac{(1 - (2+\delta) \alpha)(1 - (2+\delta) \beta)}{l^2 + \epsilon} \geqslant \frac{1 - (2+\delta)(\alpha + \beta)}{l^2 + \epsilon} \geqslant \frac{1 - (2+\delta) t}{l^2 + \epsilon} = (1 - (2+\delta)t) p. $$ Since we can choose $ \delta > 0 $ to be arbitrarily small, we in fact get $$ \| \{ F,G \} \| \geqslant (1-2t)p = s .$$ Thus we have shown that for any $ F,G \in C^\infty_c(M) $ with $ \| F - f \| + \| G - g \| \leqslant t  $ we have $ \| \{ F,G \} \| \geqslant s $. This immediately implies $$ \rho_{f,g}(s) \geqslant t = \frac{1}{2} - \frac{1}{2p} s .$$
\end{proof}


\end{document}